\documentclass[12pt, letterpaper]{article}
\usepackage{amssymb, amsmath}
\usepackage{amsthm}
\usepackage{quiver}

\title{Pathology of formal locally-trivial deformations in positive characteristic}
\author{Takuya Miyamoto\thanks{\textsc{Mathematical Sciences, the University of Tokyo, Meguro Komaba 3-8-1, Tokyo, Japan}\\
    \textit{E-mail address}: \texttt{miyamoto-takuya@g.ecc.u-tokyo.ac.jp}}}
\begin{document}
\maketitle
\begin{abstract}
An infinitesimal deformation of an algebraic variety $X$ is called (formally) locally trivial if it is Zariski-locally isomorphic to the trivial deformation, and the locally trivial deformation functor of $X$ is the subfunctor of the usual deformation functor associated with $X$ consisting of locally trivial deformations.
In this article, we construct explicit examples that are algebraic varieties in positive characteristic to show that locally trivial deformation functors do not always satisfy Schlessinger's first condition $(H_1)$, in contrast to the complex/characteristic $0$ case.  The first example is an algebraic curve and the second is a normal rational projective surface with only one rational double point. In constructing them, the characteristic $p$ can be any positive prime. The proof that they do not satisfy $(H_1)$ depends on non-liftability of certain kinds of infinitesimal deformation automorphisms, which is a phenomenon peculiar to positive characteristic. In particular, this provides a negative answer to a question posed by H. Flenner and S. Kosarew.

\end{abstract}
\tableofcontents
\

\newtheorem{theorem}{Theorem}[section]
\newtheorem{corollary}{Corollary}[theorem]
\newtheorem{lemma}[theorem]{Lemma}
\newtheorem{definition}{Definition}
\newtheorem{deflemma}[theorem]{Definition-Lemma}
\newtheorem{remark}{Remark}

\section{Introduction}
We will consider (formal) locally trivial deformations and locally trivial deformation functors. An infinitesimal deformation of an algebraic variety $X$ is called (formally) locally trivial if it is Zariski-locally isomorphic to the trivial deformation. The locally trivial deformation functor of $X$ is the subfunctor of the usual deformation functor associated with $X$ consisting of locally trivial deformations. It has been shown that, if $X$ is a complex analytic space or an algebraic scheme over a field of characteristic $0$, its locally trivial deformation functor $\operatorname{Def}_X’$ satisfies Schlessinger’s axioms $(H_1) $ and $(H_2)$, so that if it satisfies $(H_3)$ in addition, $X$ admits semiuniversal locally trivial deformation \cite[Corollary 2.6]{BGL}. This fact is fundamental and important in the theory of locally trivial deformations. Then, as pointed out in \cite[p. 630]{Fle}, the natural question arises: What if $X$ is an algebraic scheme over a field of positive characteristic $p$? The proof for the characteristic $0$ case does not apply because it involves exponentials and in particular the fraction $\frac{1}{p!}$. In \cite[Theorem 2.4.1]{Ser}, it is asserted that locally trivial deformation functors satisfy $(H_1)$ even in positive characteristic, but there seems to be a subtle gap in the proof. We answer this question in the negative by the following theorem:
\begin{theorem}
There exist a singular rational curve and a normal rational projective surface with one rational double point whose locally trivial deformation functors do not satisfy Schlessinger's condition $(H_1)$.
\end{theorem}
This is going to be proved in Theorem 3.2 and Theorem 4.2.
Such examples seem to be constructed for the first time ever. In Theorem 1.1, the characteristic $p$ can be any positive prime. We cannot expect such examples to be smooth because if they were, their locally trivial deformation functors coincide with (not necessarily locally trivial) deformation functors. The proofs that they do not satisfy $(H_1)$ are based on non-liftability of certain kinds of infinitesimal deformation automorphisms in positive characteristic as pointed out in \cite[Remark 2.9]{BGL}. We consider locally trivial deformation $X''$ that restricts to the trivial deformation $X$ along a certain surjection of Artinian algebras such that, at a singular point $P$, they do not have simultaneous trivializations so that $\mathcal{O}_{X'',P}\to \mathcal{O}_{X,P}$ becomes isomorphic to the trivial restriction morphism. Section 3 and Section 4 are independent but based on common ideas.

\section{Preliminaries}
Let $k$ be a field. The symbol $\mathcal{A}$ denotes the category of Artinian local $k$-algebras with residue field $k$ whose morphisms are local $k$-algebra homomorphisms. A $k$-scheme is called \emph{algebraic} if it is separated and of finite type over $k$. Thus, an algebraic curve is an algebraic scheme of dimension one.\\

First, let us recall some basic definitions. Let $X$ be an algebraic scheme and $A\in \operatorname{ob}\mathcal{A}.$ A cartesian diagram of algebraic $k$-schemes

% https://q.uiver.app/#q=WzAsNCxbMCwxLCJcXG9wZXJhdG9ybmFtZXtTcGVjfWsiXSxbMSwxLCJcXG9wZXJhdG9ybmFtZXtTcGVjfUEiXSxbMCwwLCJYIl0sWzEsMCwiXFxtYXRoY2Fse1h9Il0sWzIsMF0sWzAsMV0sWzIsM10sWzMsMV1d
\[\begin{tikzcd}
	X & {\mathcal{X}} \\
	{\operatorname{Spec}k} & {\operatorname{Spec}A}
	\arrow[from=1-1, to=1-2]
	\arrow[from=1-1, to=2-1]
	\arrow[from=1-2, to=2-2]
	\arrow[from=2-1, to=2-2]
\end{tikzcd}\]

is called a \emph{deformation} of $X$ over $A$ if $\mathcal{X}\to \operatorname{Spec}A$ is flat and surjective. Note that $\mathcal{X}$ and $X$ have isomorphic underlying topological spaces. For simplicity, let $\mathcal{X}$ refer to the whole above diagram. Let 
% https://q.uiver.app/#q=WzAsNCxbMCwxLCJcXG9wZXJhdG9ybmFtZXtTcGVjfWsiXSxbMSwxLCJcXG9wZXJhdG9ybmFtZXtTcGVjfUEiXSxbMCwwLCJYIl0sWzEsMCwiXFxtYXRoY2Fse1l9Il0sWzIsMF0sWzAsMV0sWzIsM10sWzMsMV1d
\[\begin{tikzcd}
	X & {\mathcal{Y}} \\
	{\operatorname{Spec}k} & {\operatorname{Spec}A}
	\arrow[from=1-1, to=1-2]
	\arrow[from=1-1, to=2-1]
	\arrow[from=1-2, to=2-2]
	\arrow[from=2-1, to=2-2]
\end{tikzcd}\]
be another deformation of $X$ over $A$. We say that two deformations $\mathcal{X}$ and $\mathcal{Y}$ are \emph{isomorphic} if there exists a commutative diagram of algebraic $k$-schemes
% https://q.uiver.app/#q=WzAsOCxbMSwyLCJcXG9wZXJhdG9ybmFtZXtTcGVjfWsiXSxbMywyLCJcXG9wZXJhdG9ybmFtZXtTcGVjfUEiXSxbMSwwLCJYIl0sWzMsMCwiXFxtYXRoY2Fse1h9Il0sWzAsMSwiWCJdLFswLDMsIlxcb3BlcmF0b3JuYW1le1NwZWN9ayJdLFsyLDMsIlxcb3BlcmF0b3JuYW1le1NwZWN9QSJdLFsyLDEsIlxcbWF0aGNhbHtZfSJdLFsyLDBdLFswLDFdLFsyLDNdLFszLDFdLFsyLDQsIiIsMSx7ImxldmVsIjoyLCJzdHlsZSI6eyJoZWFkIjp7Im5hbWUiOiJub25lIn19fV0sWzAsNSwiIiwxLHsibGV2ZWwiOjIsInN0eWxlIjp7ImhlYWQiOnsibmFtZSI6Im5vbmUifX19XSxbMSw2LCIiLDEseyJsZXZlbCI6Miwic3R5bGUiOnsiaGVhZCI6eyJuYW1lIjoibm9uZSJ9fX1dLFs0LDVdLFs1LDZdLFs0LDddLFszLDcsIlxcc2ltZXEiXSxbNyw2XV0=
\[\begin{tikzcd}
	& X && {\mathcal{X}} \\
	X && {\mathcal{Y}} \\
	& {\operatorname{Spec}k} && {\operatorname{Spec}A} \\
	{\operatorname{Spec}k} && {\operatorname{Spec}A}
	\arrow[from=1-2, to=1-4]
	\arrow[Rightarrow, no head, from=1-2, to=2-1]
	\arrow[from=1-2, to=3-2]
	\arrow["\simeq","\sigma"', from=1-4, to=2-3]
	\arrow[from=1-4, to=3-4]
	\arrow[from=2-1, to=2-3]
	\arrow[from=2-1, to=4-1]
	\arrow[from=2-3, to=4-3]
	\arrow[from=3-2, to=3-4]
	\arrow[Rightarrow, no head, from=3-2, to=4-1]
	\arrow[Rightarrow, no head, from=3-4, to=4-3]
	\arrow[from=4-1, to=4-3]
\end{tikzcd}\]

where $\sigma$ is an isomorphism. Let $\operatorname{Def}_X(A)$ denote the set $$\{\mathrm{deformations \ of\  }X\mathrm{\ over \ }A\}/\mathrm{isomorhisms}.$$
Deformations are compatible with pullbacks, and $\operatorname{Def}_X$ becomes a covariant functor $\mathcal{A}\to (\mathrm{sets}),$ which we call the \emph{deformation functor} of $X$. A deformation $\mathcal{X}$ of $X$ over $A$ is \emph{trivial} [resp. \emph{(formally) locally trivial}] if there exists a commutative diagram of algebraic $k$-schemes 
% https://q.uiver.app/#q=WzAsNixbMCwxLCJcXG9wZXJhdG9ybmFtZXtTcGVjfWsiXSxbMSwxLCJcXG9wZXJhdG9ybmFtZXtTcGVjfUEiXSxbMCwwLCJYIl0sWzEsMCwiXFxtYXRoY2Fse1h9Il0sWzIsMSwiXFxvcGVyYXRvcm5hbWV7U3BlY31rIl0sWzIsMCwiWCJdLFsyLDBdLFswLDFdLFsyLDNdLFszLDFdLFsxLDRdLFs1LDRdLFszLDVdXQ==
\[\begin{tikzcd}
	X & {\mathcal{X}} & X \\
	{\operatorname{Spec}k} & {\operatorname{Spec}A} & {\operatorname{Spec}k}
	\arrow[from=1-1, to=1-2]
	\arrow[from=1-1, to=2-1]
	\arrow[from=1-2, to=1-3]
	\arrow[from=1-2, to=2-2]
	\arrow[from=1-3, to=2-3]
	\arrow[from=2-1, to=2-2]
	\arrow[from=2-2, to=2-3]
\end{tikzcd}\]
where the two squares are cartesian [resp. if there exists an open cover $\{U_i\}$ for $X$ such that each deformation $\mathcal{X}|_{U_i}$ is trivial]. Locally trivial deformations (mod isomorphisms) consist a subfunctor  $\operatorname{Def}'_X$ of $\operatorname{Def}_X$, which we call the {\em locally trivial deformation functor} of $X$.
Let $F:\mathcal{A}\to (\mathrm{sets})$ be a covariant functor. we say that $F$ satisfies \emph{Schlessinger's condition $(H_1)$} if, for any cartesian diagram 
% https://q.uiver.app/#q=WzAsNCxbMCwxLCJBJyciXSxbMCwwLCJcXGJhcntBfSJdLFsxLDEsIkEiXSxbMSwwLCJBJyJdLFsxLDBdLFswLDJdLFszLDJdLFsxLDNdXQ==
\[\begin{tikzcd}
	{\overline{A}} & {A'} \\
	{A''} & A
	\arrow[from=1-1, to=1-2]
	\arrow[from=1-1, to=2-1]
	\arrow[from=1-2, to=2-2]
	\arrow[from=2-1, to=2-2]
\end{tikzcd}\]
in $\mathcal{A}$, the induced canonical map
$$F(\overline{A})\to F(A'')\times_{F(A)} F(A')$$
is surjective whenever $A''\to A$ is a \emph{small extension} (i.e. surjective and its kernel is a one-dimensional vector space). From $(H_1)$, it follows that the above map is surjective if $A''\to A$ is \emph{surjective}. Our object is to construct an algebraic $k$-scheme $X_0$ such that $\operatorname{Def}'_{X_0}$ does not satisfy $(H_1)$.

Now, let us state some auxiliary lemmas.

\begin{lemma}\emph{(=\cite[Lemma 3.3]{Sch})}
Let $A$ be a ring, $J$ nilpotent ideal in $A$, and $u: M\to N$ a homomorphism of $A$-modules, with $N$ flat over $A$. If $\tilde{u}: M/JM\to N/JN$ is an isomorphism, then $u$ is an isomorphism.
\end{lemma}

\begin{definition}
    Let $A''\to A, A'\to A$ be morphisms in $\mathcal{A}$ with $A''\to A$ surjective, $X_0$ an algebraic $k$-scheme, $X,X'',X'$ be deformations of $X_0$ over $A,A'',A'$ respectively such that the pullback of $X''$and $X'$ to $A$ is isomorphic to $X$ over $A$. Let $\overline{A}=A''\times_AA'$. Let $\alpha :\operatorname{Def}_{X_0}(\overline{A})\to \operatorname{Def}_{X_0}(A'')\times_{\operatorname{Def}_{X_0}(A)}\operatorname{Def}_{X_0}(A')$ be the canonical map and $[X''],[X']$ denote the isomorphism classes of $X'',X'$in $\operatorname{Def}_{X_0}(A''),\operatorname{Def}_{X_0}(A')$. If $\sigma: \mathcal{O}_{X''}\otimes A \to \mathcal{O}_{X'}\otimes A $ is an $A$-algebra isomorphism inducing $\operatorname{id}:\mathcal{O}_X \to \mathcal{O}_X $, then we denote $O_{X''}\times_{\sigma}O_{X'}$ the fiber product of $O_{X''}\to \mathcal{O}_{X''}\otimes A \xrightarrow[]{\sigma} \mathcal{O}_{X'}\otimes A $ and $\mathcal{O}_{X'} \to \mathcal{O}_{X'}\otimes A.$ 
    
\end{definition}

Note that $O_{X''}\times_{\sigma}O_{X'}$ as above induces a deformation over $\overline{A}$, i.e, it is flat over $\overline{A}$ by \cite[Lemma 3.4]{Sch}.

\begin{lemma}
In the above situation, for any $[\overline{X}]\in \alpha^{-1}([X''],[X'])$, $\mathcal{O}_{\bar{X}}$ is isomorphic to $\mathcal{O}_{X''}\times_{\sigma}\mathcal{O}_{X'}$ over $\bar{A}$ for some $\sigma$.

  \begin{proof} Choose isomorphisms $f'':\mathcal{O}_{\bar{X}}\otimes A'' \to \mathcal{O}_{X''}$ and $g':\mathcal{O}_{\bar{X}}\otimes A' \to \mathcal{O}_X$. They reduce to $f:\mathcal{O}_{\bar{X}}\otimes A \to \mathcal{O}_{X''}\otimes A$ and $g :\mathcal{O}_{\bar{X}}\otimes A \to \mathcal{O}_X \otimes A$. Set $\sigma :=g\circ f^{-1}$. Then, we have a canonical morphism $\mathcal{O}_{\bar{X}}\to \mathcal{O}_{X''}\times_{\sigma}\mathcal{O}_{X'}$ which is an isomorphism by \cite[Lemma 3.4]{Sch}; the assumption that $A''\to A$ is surjective is needed to apply this lemma.
\[\begin{tikzcd}
	&& {} \\
	& {\mathcal{O}_{\bar{X}}} && {\mathcal{O}_{\bar{X}} \otimes A'} \\
	{\mathcal{O}_{\bar{X}} \otimes A''} && {\mathcal{O}_{\bar{X}} \otimes A} \\
	&&& {\mathcal{O}_{X'}} \\
	{\mathcal{O}_{X''}} & {\mathcal{O}_{X''} \otimes A} & {\mathcal{O}_{X'} \otimes A} \\
	&& {}
	\arrow[from=2-2, to=2-4]
	\arrow[from=2-2, to=3-1]
	\arrow[from=2-4, to=3-3]
	\arrow["\simeq","g'"', from=2-4, to=4-4]
	\arrow[from=3-1, to=3-3]
	\arrow["\simeq"',"f''", from=3-1, to=5-1]
	\arrow["f","\simeq"', curve={height=6pt}, from=3-3, to=5-2]
	\arrow["g"',"\simeq", from=3-3, to=5-3]
	\arrow[from=4-4, to=5-3]
	\arrow[from=5-1, to=5-2]
	\arrow["\sigma"', dashed, from=5-2, to=5-3]
\end{tikzcd}\] \end{proof}
\end{lemma}

\section{Example I: Singular curve}
In this section, we assume that the base field $k$ is of characteristic $p>0$.
The goal is to construct an algebraic curve $X_0$ and to show that it does not satisfy Schlessinger's criterion $(H_1)$.\\
Let $R=k[t^p, t^{p+1}]\subset k[t]$, $S=k[s^{2p+1},s^{2p+2}]\subset k[s]$, $U=\operatorname{Spec}R, V=\operatorname{Spec}S$. We glue $U$ and $V$ along $T:=k[t,t^{-1}]=k[s,s^{-1}]$ by $t=s^{-1}$, to obtain $$X_0=U\cup V.$$
As the images of $R$ and $S$ in $T$ generate $T$, we see that $X_0$ is separated.

Let $A:=k[\varepsilon]$, $A'=A''=k[\lambda]$ and $\pi :k[\lambda] \to k[\varepsilon] $ be defined by $\varepsilon^2=0, \lambda^{p+1}=0$, $\lambda \mapsto \varepsilon$.

Let us consider the following cartesian diagram in $\mathcal{A}$:

% https://q.uiver.app/#q=WzAsNCxbMCwxLCJrW1xcbGFtYmRhXSJdLFsxLDEsImtbXFx2YXJlcHNpbG9uXSJdLFsxLDAsImtbXFxsYW1iZGFdIl0sWzAsMCwiXFxiYXJ7QX0iXSxbMCwxLCJcXHBpIiwyXSxbMiwxLCJcXHBpIl0sWzMsMl0sWzMsMF1d
\[\begin{tikzcd}
	{\overline{A}} & {k[\lambda]} \\
	{k[\lambda]} & {k[\varepsilon]}
	\arrow[from=1-1, to=1-2]
	\arrow[from=1-1, to=2-1]
	\arrow["\pi", from=1-2, to=2-2]
	\arrow["\pi"', from=2-1, to=2-2]
\end{tikzcd}\]

To show that $X_0$ do not satisfy $(H_1)$, it suffices to construct locally trivial deformations of $X_0$ over $k[\varepsilon]$, $k[\lambda]$, $k[\lambda]$ which fit in the above diagram and which do not lift to any locally trivial deformation over $\overline{A}$.
Let $X:=X_0[\varepsilon]$ and $X'=X_0[\lambda]$ be trivial deformations and $X_0[\lambda]\to X_0[\varepsilon]$ be defined by $\pi$. To construct $X''$, we glue $R[\lambda]:=R\otimes_k k[\lambda]=k[t^p,t^{p+1},\lambda]$ and $S[\lambda]=k[s^{2p+1},s^{2p+2},\lambda]$ by $$k[t,t^{-1},\lambda]\to k[s,s^{-1},\lambda], (t,\lambda)\mapsto (s^{-1}+\lambda s^p,\lambda).$$ 
This is indeed an isomorphism because $s^{-1}+\lambda s^p=s^{-1}(1+\lambda s^{p+1})$ and  $1+\lambda s^{p+1}$ is a unit. Thus, $X''$ is a locally trivial $k[\lambda]$-deformation. (At this point, we have not used the characteristic $p$ assumption.) Moreover, the following commutative diagram shows that $X''$ pulls back to $X$:

\[\begin{tikzcd}
	{k[t^p,t^{p+1},\lambda]} & {k[t,t^{-1},\lambda]} &&& {k[s,s^{-1},\lambda]} & {k[s^{2p+1},s^{2p+2},\lambda]} \\
	{k[t^p,t^{p+1},\varepsilon]} & {k[t,t^{-1},\varepsilon]} &&& {k[s,s^{-1},\varepsilon]} & {k[s^{2p+1},s^{2p+2},\varepsilon]} \\
	{k[t^p,t^{p+1},\varepsilon]} & {k[t,t^{-1},\varepsilon]} &&& {k[s,s^{-1},\varepsilon]} & {k[s^{2p+1},s^{2p+2},\varepsilon]}
	\arrow[hook, from=1-1, to=1-2]
	\arrow["{\operatorname{id}\otimes \pi}", from=1-1, to=2-1]
	\arrow["{(t,\lambda)\mapsto (s^{-1}+\lambda s^{p},\lambda)}","\simeq"', from=1-2, to=1-5]
	\arrow["{\operatorname{id}\otimes \pi}", from=1-2, to=2-2]
	\arrow["{\operatorname{id}\otimes \pi}", from=1-5, to=2-5]
	\arrow[hook', from=1-6, to=1-5]
	\arrow["{\operatorname{id}\otimes \pi}", from=1-6, to=2-6]
	\arrow[hook, from=2-1, to=2-2]
	\arrow["\alpha","\simeq"', from=2-1, to=3-1]
	\arrow["{(t,\varepsilon)\mapsto (s^{-1}+\varepsilon s^{p},\varepsilon)}","\simeq"', from=2-2, to=2-5]
	\arrow["\beta","\simeq"', from=2-2, to=3-2]
	\arrow[Rightarrow, no head, from=2-5, to=3-5]
	\arrow[hook', from=2-6, to=2-5]
	\arrow[Rightarrow, no head, from=2-6, to=3-6]
	\arrow[hook, from=3-1, to=3-2]
	\arrow["{(t,\varepsilon)\mapsto (s^{-1},\varepsilon)}","\simeq"', from=3-2, to=3-5]
	\arrow[hook', from=3-6, to=3-5]
\end{tikzcd}\]

where $$\alpha:(t^p,t^{p+1},\varepsilon)\mapsto (t^p,t^{p+1}+\varepsilon,\varepsilon), $$ $$ \beta : (t,\varepsilon)\to (t+\frac{\varepsilon}{t^p},\varepsilon).$$

To see that $\alpha$ is well-defined, note that $\beta$ (which is well-defined) takes $k[t^p,t^{p+1},\varepsilon]$ into $k[t^p,t^{p+1},\varepsilon]$, using the characteristic $p$ assumption. Again, $\alpha$ and $\beta$ are isomorphisms by Lemma 2.1. The above diagram shows that we can identify $\mathcal{O}_{X''}\to \mathcal{O}_{X''}\otimes k[\varepsilon]$ with $f:\mathcal{O}_{X''}\to \mathcal{O}_{X}$ where $f$ is defined by gluing the vertical morphisms of
% https://q.uiver.app/#q=WzAsOCxbMCwwLCJrW3RecCx0XntwKzF9LFxcbGFtYmRhXSJdLFsxLDAsImtbdCx0XnstMX0sXFxsYW1iZGFdIl0sWzQsMCwia1tzLHNeey0xfSxcXGxhbWJkYV0iXSxbNSwwLCJrW3NeezJwKzF9LHNeezJwKzJ9LFxcbGFtYmRhXSJdLFswLDEsImtbdF5wLHRee3ArMX0sXFx2YXJlcHNpbG9uXSJdLFsxLDEsImtbdCx0XnstMX0sXFx2YXJlcHNpbG9uXSJdLFs1LDEsImtbc157MnArMX0sc157MnArMn0sXFx2YXJlcHNpbG9uXSJdLFs0LDEsImtbcyxzXnstMX0sXFx2YXJlcHNpbG9uXSJdLFswLDEsIiIsMSx7InN0eWxlIjp7InRhaWwiOnsibmFtZSI6Imhvb2siLCJzaWRlIjoidG9wIn19fV0sWzMsMiwiIiwxLHsic3R5bGUiOnsidGFpbCI6eyJuYW1lIjoiaG9vayIsInNpZGUiOiJib3R0b20ifX19XSxbMSwyLCIodCxcXGxhbWJkYSlcXG1hcHN0byAoc157LTF9K1xcbGFtYmRhIHNee3B9LFxcbGFtYmRhKSJdLFs0LDUsIiIsMCx7InN0eWxlIjp7InRhaWwiOnsibmFtZSI6Imhvb2siLCJzaWRlIjoidG9wIn19fV0sWzMsNiwiXFxvcGVyYXRvcm5hbWV7aWR9XFxvdGltZXMgXFxwaSJdLFsyLDcsIlxcb3BlcmF0b3JuYW1le2lkfVxcb3RpbWVzIFxccGkiXSxbNiw3LCIiLDEseyJzdHlsZSI6eyJ0YWlsIjp7Im5hbWUiOiJob29rIiwic2lkZSI6ImJvdHRvbSJ9fX1dLFs1LDcsIih0LFxcbGFtYmRhKVxcbWFwc3RvIChzXnstMX0sXFx2YXJlcHNpbG9uKSJdLFswLDQsIlxcYWxwaGFcXGNpcmMgKFxcb3BlcmF0b3JuYW1le2lkfVxcb3RpbWVzIFxccGkpIl0sWzEsNSwiXFxiZXRhXFxjaXJjIChcXG9wZXJhdG9ybmFtZXtpZH1cXG90aW1lcyBcXHBpKSJdXQ==
\[\begin{tikzcd}
	{k[t^p,t^{p+1},\lambda]} & {k[t,t^{-1},\lambda]} &&& {k[s,s^{-1},\lambda]} & {k[s^{2p+1},s^{2p+2},\lambda]} \\
	{k[t^p,t^{p+1},\varepsilon]} & {k[t,t^{-1},\varepsilon]} &&& {k[s,s^{-1},\varepsilon]} & {k[s^{2p+1},s^{2p+2},\varepsilon]}
	\arrow[hook, from=1-1, to=1-2]
	\arrow["{\alpha\circ (\operatorname{id}\otimes \pi)}", from=1-1, to=2-1]
	\arrow["{(t,\lambda)\mapsto (s^{-1}+\lambda s^{p},\lambda)}", from=1-2, to=1-5]
	\arrow["{\beta\circ (\operatorname{id}\otimes \pi)}", from=1-2, to=2-2]
	\arrow["{\operatorname{id}\otimes \pi}", from=1-5, to=2-5]
	\arrow[hook', from=1-6, to=1-5]
	\arrow["{\operatorname{id}\otimes \pi}", from=1-6, to=2-6]
	\arrow[hook, from=2-1, to=2-2]
	\arrow["{(t,\lambda)\mapsto (s^{-1},\varepsilon)}", from=2-2, to=2-5]
	\arrow[hook', from=2-6, to=2-5]
\end{tikzcd}\]

\begin{lemma} If $\sigma: \mathcal{O}_X\to \mathcal{O}_X¥$ is a $k[\varepsilon]$-algebra isomorphism inducing $\operatorname{id}:\mathcal{O}_{X_0} \to \mathcal{O}_{X_0} $, then we have 
$$\sigma_U(t^p)=t^p,$$
$$\sigma_U(t^{p+1})-t^{p+1}\in \varepsilon t^p(k[t^p,t^{p+1},\varepsilon])+ \varepsilon t^{p+1}(k[t^p,t^{p+1},\varepsilon])\subset k[t^p,t^{p+1},\varepsilon].$$

\begin{proof}
First, note that the sections of $\mathcal{O}_{X}$ over non-empty open subsets can be identified with their image in the stalk of $\mathcal{O}_{X}$ at the generic point $\eta$. In particular, $\sigma(t)$ is of the form $t+ x\varepsilon, x\in \mathcal{O}_{X_0,\eta}$ so that $\sigma(t^p)=(t+ x\varepsilon)^p=t^p+pt^{p-1}x\varepsilon=t^p.$ On the other hand, as $\sigma_U(t^{p+1})\in H^0(U,\mathcal{O}_X)$, it is of the form $$\sigma(t^{p+1})=t^{p+1}+\varepsilon(a_0+a_pt^p+a_{p+1}t^{p+1}+\mathrm{higher \ order \ terms}).$$
We obtain $$\sigma(t)=\sigma(t^{p+1}/t^p)=t+\varepsilon(a_0t^{-p}+a_p+a_{p+1}t+\mathrm{higher \ order \ terms}).$$ We can replace $t, U$ by $s, V$ in the above argument and $$\sigma(s^{2p})=s^{2p},$$ $$\sigma(s^{2p+1})=s^{2p+1}+\varepsilon(b_0+b_{2p+1}s^{2p+1}+b_{2p+2}s^{2p+2}+b_{4p+2}s^{4p+2}+\mathrm{higher \ order \ terms}).$$ We can compute the inverse to the latter expression as $$\sigma(s^{2p+1})^{-1}=\frac{1}{s^{2p+1}}-\frac{\varepsilon}{s^{4p+2}}(b_0+b_{2p+1}s^{2p+1}+b_{2p+2}s^{2p+2}+b_{4p+2}s^{4p+2}+\mathrm{higher \ order \ terms}).$$ We can combine these equations to obtain another expression for $\sigma(t)$:
$$
        \begin{aligned}
\sigma(t)&=\sigma(s^{-1})=\sigma(\frac{s^{2p}}{s^{2p+1}}) \\ &=\frac{s^{2p}}{s^{2p+1}}-\frac{s^{2p}\varepsilon}{s^{4p+2}}(b_0+b_{2p+1}s^{2p+1}+b_{2p+2}s^{2p+2}+b_{4p+2}s^{4p+2}+\mathrm{higher \ order \ terms}) \\
&=t+\varepsilon(b_0t^{2p+2}+b_{2p+1}t+b_{2p+2}+b_{4p+2}t^{-2p}+\mathrm{lower \ order \ terms}). \\
        \end{aligned}
$$
In particular, we have $a_0=0$ and the conclusion follows.

\end{proof}
\end{lemma}

The following is the main theorem in this section.
\begin{theorem}
Let $A,A',A'',\overline{A},X_0,X,X'$ and $X''$ as constructed above. Then, the pair $(X'', X')$ does not lift to any locally trivial deformation over $\overline{A}$. In particular, $X_0$ is a curve in positive characteristic whose locally trivial deformation functor does not satisfy Schlessinger's condition $(H_1)$.
\end{theorem}
\begin{proof} By Lemma 2.2, it suffices to show that for any $\sigma$ as in Lemma 3.1, the fiber product $\mathcal{O}_{X''}\times_\sigma \mathcal{O}_{X'}$ is not locally trivial. So, let us suppose $\mathcal{O}_{X''}\times_\sigma \mathcal{O}_{X'}$ is locally trivial. We are going to show this leads to contradiction. Let us denote by $R_t$ [resp. $\overline{R_t}$] the stalk of $\mathcal{O}_{X_0}$ [$\mathcal{O}_{X''}\times_\sigma \mathcal{O}_{X'}$] at the point $t=0$. $R_t$ is the localization of $k[t^p,t^{p+1}]$ at $(t^p,t^{p+1})$. By the triviality, we have a $k$-section $g: R_t\to \overline{R_t}$ for the canonical morphism $\overline{R_t}\to R_t$. We have the following commutative diagram:

% https://q.uiver.app/#q=WzAsNixbMSwxLCJcXGJhcntSfV90Il0sWzMsMiwiUl90XFxvdGltZXMga1tcXHZhcmVwc2lsb25dIl0sWzEsMiwiUl90XFxvdGltZXNfa2tbXFxsYW1iZGFdIl0sWzMsMSwiUl90XFxvdGltZXNfa2tbXFxsYW1iZGFdIl0sWzAsMCwiUl90Il0sWzIsMiwiUl90XFxvdGltZXMga1tcXHZhcmVwc2lsb25dIl0sWzMsMSwiXFxvcGVyYXRvcm5hbWV7aWR9XFxvdGltZXMgXFxwaSIsMl0sWzAsMiwiXFxvcGVyYXRvcm5hbWV7cHJ9XzEiXSxbNCwwLCJnIiwxXSxbNCwzLCJcXG9wZXJhdG9ybmFtZXtwcn1fMlxcY2lyYyBnIiwwLHsiY3VydmUiOi0yLCJzdHlsZSI6eyJib2R5Ijp7Im5hbWUiOiJkYXNoZWQifX19XSxbNCwyLCJcXG9wZXJhdG9ybmFtZXtwcn1fMVxcY2lyYyBnIiwyLHsiY3VydmUiOjMsInN0eWxlIjp7ImJvZHkiOnsibmFtZSI6ImRhc2hlZCJ9fX1dLFswLDMsIlxcb3BlcmF0b3JuYW1le3ByfV8yIl0sWzIsNSwiZl90Il0sWzUsMSwiXFxzaWdtYV90Il1d
\[\begin{tikzcd}
	{R_t} \\
	& {\overline{R_t}} && {R_t\otimes_kk[\lambda]} \\
	& {R_t\otimes_kk[\lambda]} & {R_t\otimes k[\varepsilon]} & {R_t\otimes k[\varepsilon]}
	\arrow["g"{description}, from=1-1, to=2-2]
	\arrow["{\operatorname{pr}_2\circ g}", curve={height=-12pt}, dashed, from=1-1, to=2-4]
	\arrow["{\operatorname{pr}_1\circ g}"', curve={height=18pt}, dashed, from=1-1, to=3-2]
	\arrow["{\operatorname{pr}_2}", from=2-2, to=2-4]
	\arrow["{\operatorname{pr}_1}", from=2-2, to=3-2]
	\arrow["{\operatorname{id}\otimes \pi}"', from=2-4, to=3-4]
	\arrow["{f_t}", from=3-2, to=3-3]
	\arrow["{\sigma_t}", from=3-3, to=3-4]
\end{tikzcd}\]

where $f_t,\sigma_t$ is the stalk of $f,\sigma$ at $t=0$. By tensoring $\operatorname{id}_{k[\lambda]}$ with $\operatorname{pr}_i\circ g$ and factoring $f$ as $\alpha \circ (\operatorname{id}\otimes \pi)$ at $t=0$,

% https://q.uiver.app/#q=WzAsNixbMywyLCJSX3RcXG90aW1lc19ra1tcXHZhcmVwc2lsb25dIl0sWzAsMiwiUl90XFxvdGltZXNfa2tbXFxsYW1iZGFdIl0sWzMsMCwiUl90XFxvdGltZXNfa2tbXFxsYW1iZGFdIl0sWzAsMCwiUl90XFxvdGltZXNfa2tbXFxsYW1iZGFdIl0sWzEsMiwiUl90XFxvdGltZXNfa2tbXFx2YXJlcHNpbG9uXSJdLFsyLDIsIlJfdFxcb3RpbWVzX2trW1xcdmFyZXBzaWxvbl0iXSxbMiwwLCJcXG9wZXJhdG9ybmFtZXtpZH1cXG90aW1lcyBcXHBpIiwyXSxbMywyLCIoXFxvcGVyYXRvcm5hbWV7cHJ9XzJcXGNpcmMgZyApXFxvdGltZXMgXFxvcGVyYXRvcm5hbWV7aWR9X3trW1xcbGFtYmRhXX0iXSxbMywxLCIoXFxvcGVyYXRvcm5hbWV7cHJ9XzFcXGNpcmMgZyApXFxvdGltZXMgXFxvcGVyYXRvcm5hbWV7aWR9X3trW1xcbGFtYmRhXX0iLDJdLFsxLDQsIlxcb3BlcmF0b3JuYW1le2lkfVxcb3RpbWVzIFxccGkiLDJdLFs0LDUsIlxcYWxwaGFfdCIsMl0sWzUsMCwiXFxzaWdtYV90IiwyXV0=
\[\begin{tikzcd}
	{R_t\otimes_kk[\lambda]} &&& {R_t\otimes_kk[\lambda]} \\
	\\
	{R_t\otimes_kk[\lambda]} & {R_t\otimes_kk[\varepsilon]} & {R_t\otimes_kk[\varepsilon]} & {R_t\otimes_kk[\varepsilon]}
	\arrow["{(\operatorname{pr}_2\circ g )\otimes \operatorname{id}_{k[\lambda]}}", from=1-1, to=1-4]
	\arrow["{(\operatorname{pr}_1\circ g )\otimes \operatorname{id}_{k[\lambda]}}"', from=1-1, to=3-1]
	\arrow["{\operatorname{id}\otimes \pi}"', from=1-4, to=3-4]
	\arrow["{\operatorname{id}\otimes \pi}"', from=3-1, to=3-2]
	\arrow["{\alpha_t}"', from=3-2, to=3-3]
	\arrow["{\sigma_t}"', from=3-3, to=3-4]
\end{tikzcd}\]

Here, $(\operatorname{pr}_i\circ g )\otimes \operatorname{id}_{k[\lambda]}$ are $k[\lambda]$-isomorhisms by lemma 2.1. Hence, $\sigma_t\circ \alpha_t$ lifts to $\tilde{\alpha}_t$, a $k[\lambda]$-automorphism of $R_t\otimes_kk[\lambda]$:
% https://q.uiver.app/#q=WzAsNSxbMywyLCJSX3RcXG90aW1lc19ra1tcXHZhcmVwc2lsb25dIl0sWzAsMiwiUl90XFxvdGltZXNfa2tbXFxsYW1iZGFdIl0sWzMsMCwiUl90XFxvdGltZXNfa2tbXFxsYW1iZGFdIl0sWzEsMiwiUl90XFxvdGltZXNfa2tbXFx2YXJlcHNpbG9uXSJdLFsyLDIsIlJfdFxcb3RpbWVzX2trW1xcdmFyZXBzaWxvbl0iXSxbMiwwLCJcXG9wZXJhdG9ybmFtZXtpZH1cXG90aW1lcyBcXHBpIiwyXSxbMSwzLCJcXG9wZXJhdG9ybmFtZXtpZH1cXG90aW1lcyBcXHBpIiwyXSxbMSwyLCJcXHRpbGRle1xcYWxwaGF9X3QiLDIseyJjdXJ2ZSI6LTQsInN0eWxlIjp7ImJvZHkiOnsibmFtZSI6ImRhc2hlZCJ9fX1dLFszLDQsIlxcYWxwaGFfdCIsMl0sWzQsMCwiXFxzaWdtYV90IiwyXV0=
\[\begin{tikzcd}
	&&& {R_t\otimes_kk[\lambda]} \\
	\\
	{R_t\otimes_kk[\lambda]} & {R_t\otimes_kk[\varepsilon]} & {R_t\otimes_kk[\varepsilon]} & {R_t\otimes_kk[\varepsilon]}
	\arrow["{\operatorname{id}\otimes \pi}"', from=1-4, to=3-4]
	\arrow["{\widetilde{\alpha}_t}"', curve={height=-24pt}, dashed, from=3-1, to=1-4]
	\arrow["{\operatorname{id}\otimes \pi}"', from=3-1, to=3-2]
	\arrow["{\alpha_t}"', from=3-2, to=3-3]
	\arrow["{\sigma_t}"', from=3-3, to=3-4]
\end{tikzcd}\]

Now, by lemma 3.1, $t^p\otimes 1, t^{p+1}\otimes 1 \in R_t\otimes_kk[\lambda]$ go along the above diagram as

% https://q.uiver.app/#q=WzAsNSxbMywyLCJ0XnBcXG90aW1lczEiXSxbMCwyLCJ0XnBcXG90aW1lczEiXSxbMywwLCJ0XnBcXG90aW1lczEreCJdLFsxLDIsInRecFxcb3RpbWVzMSJdLFsyLDIsInRecFxcb3RpbWVzMSJdLFsyLDAsIlxcb3BlcmF0b3JuYW1le2lkfVxcb3RpbWVzIFxccGkiLDIseyJzdHlsZSI6eyJ0YWlsIjp7Im5hbWUiOiJtYXBzIHRvIn19fV0sWzEsMywiXFxvcGVyYXRvcm5hbWV7aWR9XFxvdGltZXMgXFxwaSIsMix7InN0eWxlIjp7InRhaWwiOnsibmFtZSI6Im1hcHMgdG8ifX19XSxbMSwyLCJcXHRpbGRle1xcYWxwaGF9X3QiLDIseyJjdXJ2ZSI6LTQsInN0eWxlIjp7InRhaWwiOnsibmFtZSI6Im1hcHMgdG8ifX19XSxbMyw0LCJcXGFscGhhX3QiLDIseyJzdHlsZSI6eyJ0YWlsIjp7Im5hbWUiOiJtYXBzIHRvIn19fV0sWzQsMCwiXFxzaWdtYV90IiwyLHsic3R5bGUiOnsidGFpbCI6eyJuYW1lIjoibWFwcyB0byJ9fX1dXQ==
\[\begin{tikzcd}
	&&& {t^p\otimes1+x} \\
	\\
	{t^p\otimes1} & {t^p\otimes1} & {t^p\otimes1} & {t^p\otimes1}
	\arrow["{\operatorname{id}\otimes \pi}"', maps to, from=1-4, to=3-4]
	\arrow["{\tilde{\alpha}_t}"', curve={height=-24pt}, maps to, from=3-1, to=1-4]
	\arrow["{\operatorname{id}\otimes \pi}"', maps to, from=3-1, to=3-2]
	\arrow["{\alpha_t}"', maps to, from=3-2, to=3-3]
	\arrow["{\sigma_t}"', maps to, from=3-3, to=3-4]
\end{tikzcd}\]
% https://q.uiver.app/#q=WzAsNSxbMiwyLCJ0XntwKzF9XFxvdGltZXMxICsxXFxvdGltZXNcXHZhcmVwc2lsb24iXSxbMCwyLCJ0XntwKzF9XFxvdGltZXMxIl0sWzMsMCwidF57cCsxfVxcb3RpbWVzMSArbVxcb3RpbWVzXFxsYW1iZGEgKzFcXG90aW1lc1xcbGFtYmRhK3kiXSxbMSwyLCJ0XntwKzF9XFxvdGltZXMxIl0sWzMsMiwidF57cCsxfVxcb3RpbWVzMSArbVxcb3RpbWVzXFx2YXJlcHNpbG9uICsxXFxvdGltZXNcXHZhcmVwc2lsb24iXSxbMSwzLCJcXG9wZXJhdG9ybmFtZXtpZH1cXG90aW1lcyBcXHBpIiwyLHsic3R5bGUiOnsidGFpbCI6eyJuYW1lIjoibWFwcyB0byJ9fX1dLFszLDAsIlxcYWxwaGFfdCIsMix7InN0eWxlIjp7InRhaWwiOnsibmFtZSI6Im1hcHMgdG8ifX19XSxbMSwyLCJcXHRpbGRle1xcYWxwaGF9X3QiLDIseyJjdXJ2ZSI6LTQsInN0eWxlIjp7InRhaWwiOnsibmFtZSI6Im1hcHMgdG8ifX19XSxbMiw0LCJcXG9wZXJhdG9ybmFtZXtpZH1cXG90aW1lcyBcXHBpIiwyLHsic3R5bGUiOnsidGFpbCI6eyJuYW1lIjoibWFwcyB0byJ9fX1dLFswLDQsIlxcc2lnbWFfdCIsMix7InN0eWxlIjp7InRhaWwiOnsibmFtZSI6Im1hcHMgdG8ifX19XV0=
\[\begin{tikzcd}
	&&& {t^{p+1}\otimes1 +m\otimes\lambda +1\otimes\lambda+y} \\
	\\
	{t^{p+1}\otimes1} & {t^{p+1}\otimes1} & {t^{p+1}\otimes1 +1\otimes\varepsilon} & {t^{p+1}\otimes1 +m\otimes\varepsilon +1\otimes\varepsilon}
	\arrow["{\operatorname{id}\otimes \pi}"', maps to, from=1-4, to=3-4]
	\arrow["{\tilde{\alpha}_t}"', curve={height=-24pt}, maps to, from=3-1, to=1-4]
	\arrow["{\operatorname{id}\otimes \pi}"', maps to, from=3-1, to=3-2]
	\arrow["{\alpha_t}"', maps to, from=3-2, to=3-3]
	\arrow["{\sigma_t}"', maps to, from=3-3, to=3-4]
\end{tikzcd}\]

where $x,y\in \operatorname{Ker}(\operatorname{id}\otimes\pi), m\in (t^p,t^{p+1})\subset R_t$. There exist $x',y'\in R_t\otimes_kk[\lambda]$ such that $x=\lambda^2x',y=\lambda^2y'$.

In $R_t\otimes_kk[\lambda]$, we have $$ \begin{aligned}0&=\widetilde{\alpha}_t((t^p)^{p+1}-(t^{p+1})^p )
 \\
 &=(t^p\otimes1+\lambda^2x')^{p+1}- (t^{p+1}\otimes1 +m\otimes \lambda +1\otimes\lambda+\lambda^2y')^p.
 \end{aligned}$$

 Let $\pi': R_t\to k$ be the canonical projection. The latter element goes under $\pi'\otimes\operatorname{id} :R_t\otimes_kk[\lambda]\to k\otimes_kk[\lambda]$ to
 $$\begin{aligned} 
 &(0\otimes1+\lambda^2x'')^{p+1}- (0\otimes1 +0\otimes1+1\otimes\lambda+\lambda^2y'')^p \\
 =&-\lambda^p(1+\lambda y'')^p \\
 =&-\lambda^p\neq 0 ,\\
\end{aligned}$$
hence contradiction. Here, we used the relation $\lambda^{p+1}=0$. $x''$ and $y''$ are the images of $x',y'$ in $k\otimes_kk[\lambda]$. 
\end{proof}

\section{Example II: Normal projective surface}
In this section, we again assume that the base field $k $ is of characteristic $p>0$.
The goal is to construct a normal rational projective surface $X_0$ with one rational double point and to show that it does not satisfy Schlessinger's criterion $(H_1)$. We realize $X_0$ as a closed subscheme of $\mathbb{P}^9_k.$\\

Let $\mathbb{P}^9_k=\operatorname{Proj}k[T_0,\dots,T_9]$ be the projective $9$-space. Let $V_i$ denote the open sets defined by $T_i\neq 0$ and let $\zeta_{ij}$ denote $T_j/T_i.$ Let $R_0=k[x,y,z]$ be the $k$-algebra defined by the relation $xz=y^p$. We identify $U_0=\operatorname{Spec}R_0$ with a closed subscheme of $V_0$ under the $k$-algebra surjection $\varphi$ defined as
$$\varphi: k[\zeta_{01},\dots,\zeta_{09}]\to R_0,$$
$$\varphi(\zeta_{01})=x,\ \varphi(\zeta_{02})=x^{p(p+1)+1}y^{p+1},\ \varphi(\zeta_{03})=yz,\ \varphi(\zeta_{04})=z,\ \varphi(\zeta_{05})=y,\ $$
$$\varphi(\zeta_{06})=x^{p+1}y,\ \varphi(\zeta_{07})=x^{p^2+1}y^p, \ \varphi(\zeta_{08})=x^{p(p+1)}y^{p+1}, \ \varphi(\zeta_{09})=y^{p+1}.$$

Let $X_0$ be the closure of $U_0$ in $\mathbb{P}^9_k$ with the reduced subscheme structure. Let $U_i:=X_0\cap V_i$ for $i \ge 1$. Then, for each $i$, $U_i$ is the Spec of the subring
$$R_i:=k[\frac{1}{\varphi(\zeta_{0i})},\frac{\varphi(\zeta_{01})}{\varphi(\zeta_{0i})},\dots,\frac{\varphi(\zeta_{09})}{\varphi(\zeta_{0i})}]\subset K$$
where $K=\operatorname{Frac}R_0.$ In view of this, we have
$$R_1=k[x^{-1},x^py],$$
$$R_2=k[x^{-1},x^{-p}y^{-1}],$$
$$R_3=k[x,y^{-1}],$$
$$R_4=k[xy^{-p},y],$$
$$R_5=k[x,x^{-1},y,y^{-1}],$$
$$R_6=R_7=k[x^py,x^{-p}y^{-1},x^{-1}],$$
$$R_8=R_9=k[x,x^{-1},y^{-1}].$$
Because $R_1,\dots , R_9$ are smooth, $X_0$ is a normal integral projective surface with only one rational double point. Moreover, $X_0$ is covered by $U_0,U_1,U_2,U_3$ and $U_4$ because $R_5,\dots ,R_9$ are localizations of the other $R_i$'s.\\

Now, consider the following cartesian diagram in $\mathcal{A}$:
% https://q.uiver.app/#q=WzAsNCxbMCwwLCJcXG92ZXJsaW5le0F9Il0sWzAsMSwiQScnPWtbXFxsYW1iZGFdIl0sWzEsMCwiQSc9a1tcXGxhbWJkYV0iXSxbMSwxLCJBPWtbXFx2YXJlcHNpbG9uXSJdLFswLDFdLFsxLDMsIlxccGkiLDJdLFsyLDMsIlxccGkiXSxbMCwyXV0=
\[\begin{tikzcd}
	{\overline{A}} & {A'=k[\lambda]} \\
	{A''=k[\lambda]} & {A=k[\varepsilon]}
	\arrow[from=1-1, to=1-2]
	\arrow[from=1-1, to=2-1]
	\arrow["\pi", from=1-2, to=2-2]
	\arrow["\pi"', from=2-1, to=2-2]
\end{tikzcd}\]

where $k[\varepsilon]$ and $k[\lambda]$ are defined by $\varepsilon^p=0$ and $\lambda^{p+1}=0$ and where $\pi (\lambda)=\varepsilon.$
To show that $X_0$ do not satisfy $(H_1)$, it suffices to construct locally trivial deformations of $X_0$ over $k[\varepsilon]$, $k[\lambda]$, $k[\lambda]$ which fit in the above diagram and which do not lift to any locally trivial deformation over $\overline{A}$.
Let $X$ and $X'$ be the trivial deformations of $X_0$ over $A$ and $A'$, respectively and let $X\to X'$ defined by $\pi$.
Next, we would like to construct a locally trivial deformation $X''$ over $A''$ via gluing. Let 
$$R_{01}=k[x,y,x^{-1}],$$
$$R_{12}=k[x^py,x^{-p}y^{-1},x^{-1}],$$
$$R_{23}=k[x,x^{-1},y^{-1}],$$
$$R_{34}=R_{43}=k[x,y,y^{-1}],$$
$$R_{04}=k[xy^{-p},x^{-1}y^p,y],$$
$$R_{02}=R_{03}=R_{13}=R_{14}=R_{24}=k[x,x^{-1},y,y^{-1}].$$
We have $U_i \cap U_j=\operatorname{Spec}R_{ij}$ for $0\le i < j\le 4$ and $U_i \cap U_j  \cap U_k = \operatorname{Spec}R_{02}$ for $0\le i < j < k \le 4$.
Let us define $k[\lambda]-$automorphisms as 
$$\psi_{01}:R_{01}[\lambda]\to R_{01}[\lambda],\ (x,y)\mapsto (x,y+\lambda),$$
$$\psi_{12}=\operatorname{id}:R_{12}[\lambda]\to R_{12}[\lambda], $$
$$\psi_{23}=\operatorname{id}:R_{23}[\lambda]\to R_{23}[\lambda], $$
$$\psi_{43}:R_{43}[\lambda]\to R_{43}[\lambda],\  (x,y)\mapsto (x(1-\frac{\lambda}{y})^p,y),$$
$$\psi_{04}:R_{04}[\lambda]\to R_{04}[\lambda],\  (xy^{-p},y)\mapsto (xy^{-p},y+\lambda),$$
$$\psi_{02}=\psi_{03}:R_{02}[\lambda]\to R_{02}[\lambda],\ (x,y)\mapsto (x,y+\lambda),$$
$$\psi_{13}=\operatorname{id}:R_{13}[\lambda]\to R_{13}[\lambda],$$
$$\psi_{14}=\psi_{24}:R_{14}[\lambda]\to R_{14}[\lambda], \ (x,y)\mapsto (x(1+\frac{\lambda}{y})^p,y).$$

The $\psi_{ij}$'s satisfy the cocycle conditions such that

$$(\operatorname{id}_{R_{02}}\otimes_{R_{jk}} \psi_{jk}) \circ (\operatorname{id}_{R_{02}}\otimes_{R_{ij}} \psi_{ij}) = \operatorname{id}_{R_{02}}\otimes_{R_{ik}}\psi_{ik}.$$

Let us glue the $U_i[\lambda]$ as
$$U_j[\lambda]=\operatorname{Spec}R_j[\lambda]\supset \operatorname{Spec}R_{ij}[\lambda]\overset{\psi_{ij}^{\sharp}}{\longrightarrow}\operatorname{Spec}R_{ij}[\lambda]\subset\operatorname{Spec}R_i[\lambda] =U_i[\lambda] $$
for every $0\le i <j\le 4$, where $\psi_{ij}^\sharp =\operatorname{Spec}\psi_{ij}$. We obtain a locally trivial deformation $X''$ over $A''.$ Note that $X''$ becomes isomorphic to $X$ when restricted to $A$. To see this, we construct an isomorphism of sheaves of $A$-algebras $$\rho : \mathcal{O}_{X''}\otimes_{A''}A\to \mathcal{O}_X=\mathcal{O}_{X_0}\otimes_k A.$$  First, note that $\psi_{43},\psi_{14}$ and $\psi_{24}$ become trivial when restricted to $A$. Let $\phi:R_0[\varepsilon]\to R_0[\varepsilon],(x,y,z)\to (x,y+\varepsilon,z).$ This is well-defined because $xz-(y+\varepsilon)^p=xz-y^p$ by the relations $p=\varepsilon^p=0$. For every $i=1,2,3,4,$ we have
$$\psi_{0i}\otimes_{A''}\operatorname{id}_{A}=\operatorname{id}_{R_{0i}} \otimes_{R_0}\phi:R_{0i}[\varepsilon]\to R_{0i}[\varepsilon].$$
This means that if we define $\rho_i=\rho|_{U_i}: R_i[\varepsilon]\to R_i[\varepsilon]$ for $0\le i \le 4$ by
$$\rho_0=\phi$$
and
$$\rho_i=\operatorname{id}_{R_i[\varepsilon]}$$
for $1\le i \le 4,$ they glue together to yield to an isomorphism of sheaves $\rho : \mathcal{O}_{X''}\otimes_{A''}A\to \mathcal{O}_X$.

\begin{lemma}
Let $\sigma: \mathcal{O}_X\to \mathcal{O}_X$ be a $k[\varepsilon]$-automorphism that induces a deformation automorphism $X\to X.$ Then, we have
$$\sigma(y)-y\in yR_0\otimes \varepsilon k[\varepsilon]+ zR_0\otimes \varepsilon k[\varepsilon]+R_0\otimes \varepsilon^2 k[\varepsilon],$$
$$\sigma(z)-z\in zR_0\otimes\varepsilon k[\varepsilon].$$
\begin{proof}
As $\sigma(x^{-p}y^{-1})\in H^0(U_2[\varepsilon],\mathcal{O}_X),$ we can write as
$$\sigma(x^{-p}y^{-1})=x^{-p}y^{-1}+\varepsilon f,\  f\in R_2[\varepsilon].$$
Similarly, $\sigma(x^{-1})\in H^0(U_1[\varepsilon],\mathcal{O}_X)\cap H^0(U_2[\varepsilon],\mathcal{O}_X)=(H^0(U_1,\mathcal{O}_{X_0})\cap H^0(U_2,\mathcal{O}_{X_0}))\otimes_k k[\varepsilon]=k[x^{-1},\varepsilon]$ so that
$$ \sigma(x^{-1})=x^{-1}+\varepsilon g,\  g\in k[x^{-1},\varepsilon].$$
Because $p=\varepsilon^p=0,$ we have
$$\sigma(x^{-p})=x^{-p},\ \sigma(x^{p})=x^p,$$
$$\sigma(y^{-1})=\sigma(x^p\cdot x^{-p}y^{-1})=y^{-1}+\varepsilon x^pf,$$
$$\sigma(y)\equiv y-\varepsilon x^py^2f \ \mathrm{mod} \ \varepsilon^2.$$
As $x^py^2R_2\cap R_0=ky+kz+ky^2+kxy^2+kx^2y^2$ if $p=2$ and $x^py^2R_2\cap R_0=ky+ky^2+kxy^2+kx^2y^2+\dots+kx^py^2$ if $p\ge 3$, the first inclusion follows.
Again because $p=\varepsilon^p=0,$ we have
$$\sigma(y^p)=y^p,$$
so that $$\sigma(z)=\sigma(x^{-1}y^p)=x^{-1}y^p+\varepsilon y^pg.$$
As $y^pk[x^{-1}]\cap R_0=kx^{-1}y^p+ky^p=kz+kxz,$ the second inclusion follows.
\end{proof}
\end{lemma}

The following is the main theorem in this section.
\begin{theorem}
Let $A,A',A'',\overline{A},X_0,X,X'$ and $X''$ as constructed above (that are different from those in Section 3). Then, the pair $(X'', X')$ does not lift to any locally trivial deformation over $\overline{A}$. In particular, $X_0$ is a normal projective rational surface with only one rational double point in positive characteristic whose locally trivial deformation functor does not satisfy Schlessinger's condition $(H_1)$.
\end{theorem}

\begin{proof} By Lemma 2.2, it suffices to show that for any $\sigma$ as in Lemma 4.1, the fiber product $\mathcal{O}_{X''}\times_\sigma \mathcal{O}_{X'}$ is not locally trivial. So, let us suppose $\mathcal{O}_{X''}\times_\sigma \mathcal{O}_{X'}$ is locally trivial. We are going to show this leads to contradiction. Let us denote by $R_P$ [resp. $\overline{R}_P$] the stalk of $\mathcal{O}_{X_0}$ [$\mathcal{O}_{X''}\times_\sigma \mathcal{O}_{X'}$] at the point $P\in U_0$ that corresponds to the maximal ideal $(x,y,z)\subset R_0$. By the triviality, we have a $k$-section $h: R_P\to \overline{R}_P$ for the canonical morphism $\overline{R}_P \to R_P$. Because we can identify $\mathcal{O}_{X''}\to \mathcal{O}_{X}$ as $\phi \circ (\operatorname{id}\otimes \pi)$ at $P$, we have the following commutative diagram:

% https://q.uiver.app/#q=WzAsNyxbMSwxLCJcXG92ZXJsaW5le1J9X1AiXSxbNCwyLCJSX1BcXG90aW1lcyBrW1xcdmFyZXBzaWxvbl0iXSxbMSwyLCJSX1BcXG90aW1lc19ra1tcXGxhbWJkYV0iXSxbNCwxLCJSX1BcXG90aW1lc19ra1tcXGxhbWJkYV0iXSxbMCwwLCJSX1AiXSxbMywyLCJSX1BcXG90aW1lcyBrW1xcdmFyZXBzaWxvbl0iXSxbMiwyLCJSX1BcXG90aW1lcyBrW1xcdmFyZXBzaWxvbl0iXSxbMywxLCJcXG9wZXJhdG9ybmFtZXtpZH1cXG90aW1lcyBcXHBpIiwyXSxbMCwyLCJcXG9wZXJhdG9ybmFtZXtwcn1fMSJdLFs0LDAsImgiLDFdLFs0LDMsIlxcb3BlcmF0b3JuYW1le3ByfV8yXFxjaXJjIGgiLDAseyJjdXJ2ZSI6LTIsInN0eWxlIjp7ImJvZHkiOnsibmFtZSI6ImRhc2hlZCJ9fX1dLFs0LDIsIlxcb3BlcmF0b3JuYW1le3ByfV8xXFxjaXJjIGgiLDIseyJjdXJ2ZSI6Mywic3R5bGUiOnsiYm9keSI6eyJuYW1lIjoiZGFzaGVkIn19fV0sWzAsMywiXFxvcGVyYXRvcm5hbWV7cHJ9XzIiXSxbNSwxLCJcXHNpZ21hX1AiXSxbNiw1LCJcXHBoaV9QIl0sWzIsNiwiXFxvcGVyYXRvcm5hbWV7aWR9XFxvdGltZXMgXFxwaSJdXQ==
\[\begin{tikzcd}
	{R_P} \\
	& {\overline{R}_P} &&& {R_P\otimes_kk[\lambda]} \\
	& {R_P\otimes_kk[\lambda]} & {R_P\otimes k[\varepsilon]} & {R_P\otimes k[\varepsilon]} & {R_P\otimes k[\varepsilon]}
	\arrow["h"{description}, from=1-1, to=2-2]
	\arrow["{\operatorname{pr}_2\circ h}", curve={height=-12pt}, dashed, from=1-1, to=2-5]
	\arrow["{\operatorname{pr}_1\circ h}"', curve={height=18pt}, dashed, from=1-1, to=3-2]
	\arrow["{\operatorname{pr}_2}", from=2-2, to=2-5]
	\arrow["{\operatorname{pr}_1}", from=2-2, to=3-2]
	\arrow["{\operatorname{id}\otimes \pi}"', from=2-5, to=3-5]
	\arrow["{\operatorname{id}\otimes \pi}", from=3-2, to=3-3]
	\arrow["{\phi_P}", from=3-3, to=3-4]
	\arrow["{\sigma_P}", from=3-4, to=3-5]
\end{tikzcd}\]

where $\phi_P,\sigma_P$ is the stalk of $\phi,\sigma$ at $P$. By tensoring $\operatorname{id}_{k[\lambda]}$ with $\operatorname{pr}_i\circ h$,

% https://q.uiver.app/#q=WzAsNixbMywyLCJSX1BcXG90aW1lc19ra1tcXHZhcmVwc2lsb25dIl0sWzAsMiwiUl9QXFxvdGltZXNfa2tbXFxsYW1iZGFdIl0sWzMsMCwiUl9QXFxvdGltZXNfa2tbXFxsYW1iZGFdIl0sWzAsMCwiUl9QXFxvdGltZXNfa2tbXFxsYW1iZGFdIl0sWzEsMiwiUl9QXFxvdGltZXNfa2tbXFx2YXJlcHNpbG9uXSJdLFsyLDIsIlJfUFxcb3RpbWVzX2trW1xcdmFyZXBzaWxvbl0iXSxbMiwwLCJcXG9wZXJhdG9ybmFtZXtpZH1cXG90aW1lcyBcXHBpIiwyXSxbMywyLCIoXFxvcGVyYXRvcm5hbWV7cHJ9XzJcXGNpcmMgaCApXFxvdGltZXMgXFxvcGVyYXRvcm5hbWV7aWR9X3trW1xcbGFtYmRhXX0iXSxbMywxLCIoXFxvcGVyYXRvcm5hbWV7cHJ9XzFcXGNpcmMgaCApXFxvdGltZXMgXFxvcGVyYXRvcm5hbWV7aWR9X3trW1xcbGFtYmRhXX0iLDJdLFsxLDQsIlxcb3BlcmF0b3JuYW1le2lkfVxcb3RpbWVzIFxccGkiLDJdLFs0LDUsIlxccGhpX1AiLDJdLFs1LDAsIlxcc2lnbWFfUCIsMl1d
\[\begin{tikzcd}
	{R_P\otimes_kk[\lambda]} &&& {R_P\otimes_kk[\lambda]} \\
	\\
	{R_P\otimes_kk[\lambda]} & {R_P\otimes_kk[\varepsilon]} & {R_P\otimes_kk[\varepsilon]} & {R_P\otimes_kk[\varepsilon]}
	\arrow["{(\operatorname{pr}_2\circ h )\otimes \operatorname{id}_{k[\lambda]}}", from=1-1, to=1-4]
	\arrow["{(\operatorname{pr}_1\circ h )\otimes \operatorname{id}_{k[\lambda]}}"', from=1-1, to=3-1]
	\arrow["{\operatorname{id}\otimes \pi}"', from=1-4, to=3-4]
	\arrow["{\operatorname{id}\otimes \pi}"', from=3-1, to=3-2]
	\arrow["{\phi_P}"', from=3-2, to=3-3]
	\arrow["{\sigma_P}"', from=3-3, to=3-4]
\end{tikzcd}\]
Here, $(\operatorname{pr}_i\circ h )\otimes \operatorname{id}_{k[\lambda]}$ are $k[\lambda]$-isomorhisms by lemma 2.1. Hence, $\sigma_P\circ \phi_P$ lifts to $\widetilde{\sigma}_P$, a $k[\lambda]$-automorphism of $R_P\otimes_kk[\lambda]$:

% https://q.uiver.app/#q=WzAsNSxbMywyLCJSX1BcXG90aW1lc19ra1tcXHZhcmVwc2lsb25dIl0sWzAsMiwiUl9QXFxvdGltZXNfa2tbXFxsYW1iZGFdIl0sWzMsMCwiUl9QXFxvdGltZXNfa2tbXFxsYW1iZGFdIl0sWzEsMiwiUl9QXFxvdGltZXNfa2tbXFx2YXJlcHNpbG9uXSJdLFsyLDIsIlJfUFxcb3RpbWVzX2trW1xcdmFyZXBzaWxvbl0iXSxbMiwwLCJcXG9wZXJhdG9ybmFtZXtpZH1cXG90aW1lcyBcXHBpIiwyXSxbMSwzLCJcXG9wZXJhdG9ybmFtZXtpZH1cXG90aW1lcyBcXHBpIiwyXSxbMSwyLCJcXHdpZGV0aWxkZXtcXHNpZ21hfV9QIiwyLHsiY3VydmUiOi00LCJzdHlsZSI6eyJib2R5Ijp7Im5hbWUiOiJkYXNoZWQifX19XSxbMyw0LCJcXHBoaV9QIiwyXSxbNCwwLCJcXHNpZ21hX1AiLDJdXQ==
\[\begin{tikzcd}
	&&& {R_P\otimes_kk[\lambda]} \\
	\\
	{R_P\otimes_kk[\lambda]} & {R_P\otimes_kk[\varepsilon]} & {R_P\otimes_kk[\varepsilon]} & {R_P\otimes_kk[\varepsilon]}
	\arrow["{\operatorname{id}\otimes \pi}"', from=1-4, to=3-4]
	\arrow["{\widetilde{\sigma}_P}"', curve={height=-24pt}, dashed, from=3-1, to=1-4]
	\arrow["{\operatorname{id}\otimes \pi}"', from=3-1, to=3-2]
	\arrow["{\phi_P}"', from=3-2, to=3-3]
	\arrow["{\sigma_P}"', from=3-3, to=3-4]
\end{tikzcd}\]

Now, by lemma 4.1, $x\otimes 1,y\otimes1,,z\otimes1 \in R_P\otimes_kk[\lambda]$ go along the above diagram as

% https://q.uiver.app/#q=WzAsNSxbMywyLCJ4XFxvdGltZXMxKyBcXHZhcmVwc2lsb24gXFxvdmVybGluZXttfSJdLFswLDIsInhcXG90aW1lczEiXSxbMywwLCJ4XFxvdGltZXMxK1xcbGFtYmRhIG0iXSxbMSwyLCJ4XFxvdGltZXMxIl0sWzIsMiwieFxcb3RpbWVzMSJdLFsyLDAsIlxcb3BlcmF0b3JuYW1le2lkfVxcb3RpbWVzIFxccGkiLDIseyJzdHlsZSI6eyJ0YWlsIjp7Im5hbWUiOiJtYXBzIHRvIn19fV0sWzEsMywiXFxvcGVyYXRvcm5hbWV7aWR9XFxvdGltZXMgXFxwaSIsMix7InN0eWxlIjp7InRhaWwiOnsibmFtZSI6Im1hcHMgdG8ifX19XSxbMSwyLCJcXHdpZGV0aWxkZXtcXHNpZ21hfV9QIiwyLHsiY3VydmUiOi00LCJzdHlsZSI6eyJ0YWlsIjp7Im5hbWUiOiJtYXBzIHRvIn19fV0sWzMsNCwiXFxwaGlfUCIsMix7InN0eWxlIjp7InRhaWwiOnsibmFtZSI6Im1hcHMgdG8ifX19XSxbNCwwLCJcXHNpZ21hX1AiLDIseyJzdHlsZSI6eyJ0YWlsIjp7Im5hbWUiOiJtYXBzIHRvIn19fV1d
\[\begin{tikzcd}
	&&& {x\otimes1+\lambda m} \\
	\\
	{x\otimes1} & {x\otimes1} & {x\otimes1} & {x\otimes1+ \varepsilon \overline{m}}
	\arrow["{\operatorname{id}\otimes \pi}"', maps to, from=1-4, to=3-4]
	\arrow["{\widetilde{\sigma}_P}"', curve={height=-24pt}, maps to, from=3-1, to=1-4]
	\arrow["{\operatorname{id}\otimes \pi}"', maps to, from=3-1, to=3-2]
	\arrow["{\phi_P}"', maps to, from=3-2, to=3-3]
	\arrow["{\sigma_P}"', maps to, from=3-3, to=3-4]
\end{tikzcd}\]

% https://q.uiver.app/#q=WzAsNSxbMywyLCJ5XFxvdGltZXMxKzFcXG90aW1lc1xcdmFyZXBzaWxvbitcXHZhcmVwc2lsb24gXFxvdmVybGluZXtuXzF9K1xcdmFyZXBzaWxvbl4yIFxcb3ZlcmxpbmV7bl8yfSJdLFswLDIsInlcXG90aW1lczEiXSxbMywwLCJ5XFxvdGltZXMxKzFcXG90aW1lc1xcbGFtYmRhK1xcbGFtYmRhIG5fMStcXGxhbWJkYV4yIG5fMiArXFxsYW1iZGFecG4nIl0sWzEsMiwieVxcb3RpbWVzMSJdLFsyLDIsInlcXG90aW1lczErMVxcb3RpbWVzXFx2YXJlcHNpbG9uIl0sWzIsMCwiXFxvcGVyYXRvcm5hbWV7aWR9XFxvdGltZXMgXFxwaSIsMix7InN0eWxlIjp7InRhaWwiOnsibmFtZSI6Im1hcHMgdG8ifX19XSxbMSwzLCJcXG9wZXJhdG9ybmFtZXtpZH1cXG90aW1lcyBcXHBpIiwyLHsic3R5bGUiOnsidGFpbCI6eyJuYW1lIjoibWFwcyB0byJ9fX1dLFsxLDIsIlxcd2lkZXRpbGRle1xcc2lnbWF9X1AiLDIseyJjdXJ2ZSI6LTQsInN0eWxlIjp7InRhaWwiOnsibmFtZSI6Im1hcHMgdG8ifX19XSxbMyw0LCJcXHBoaV9QIiwyLHsic3R5bGUiOnsidGFpbCI6eyJuYW1lIjoibWFwcyB0byJ9fX1dLFs0LDAsIlxcc2lnbWFfUCIsMix7InN0eWxlIjp7InRhaWwiOnsibmFtZSI6Im1hcHMgdG8ifX19XV0=
\[\begin{tikzcd}
	&&& {y\otimes1+1\otimes\lambda+\lambda n_1+\lambda^2 n_2 +\lambda^pn'} \\
	\\
	{y\otimes1} & {y\otimes1} & {y\otimes1+1\otimes\varepsilon} & {y\otimes1+1\otimes\varepsilon+\varepsilon \overline{n_1}+\varepsilon^2 \overline{n_2}}
	\arrow["{\operatorname{id}\otimes \pi}"', maps to, from=1-4, to=3-4]
	\arrow["{\widetilde{\sigma}_P}"', curve={height=-24pt}, maps to, from=3-1, to=1-4]
	\arrow["{\operatorname{id}\otimes \pi}"', maps to, from=3-1, to=3-2]
	\arrow["{\phi_P}"', maps to, from=3-2, to=3-3]
	\arrow["{\sigma_P}"', maps to, from=3-3, to=3-4]
\end{tikzcd}\]
% https://q.uiver.app/#q=WzAsNSxbMywyLCJ6XFxvdGltZXMxK1xcdmFyZXBzaWxvbiBcXG92ZXJsaW5le2x9Il0sWzAsMiwielxcb3RpbWVzMSJdLFszLDAsInpcXG90aW1lczErXFxsYW1iZGEgbCtcXGxhbWJkYV5wbCciXSxbMSwyLCJ6XFxvdGltZXMxIl0sWzIsMiwielxcb3RpbWVzMSJdLFsyLDAsIlxcb3BlcmF0b3JuYW1le2lkfVxcb3RpbWVzIFxccGkiLDIseyJzdHlsZSI6eyJ0YWlsIjp7Im5hbWUiOiJtYXBzIHRvIn19fV0sWzEsMywiXFxvcGVyYXRvcm5hbWV7aWR9XFxvdGltZXMgXFxwaSIsMix7InN0eWxlIjp7InRhaWwiOnsibmFtZSI6Im1hcHMgdG8ifX19XSxbMSwyLCJcXHdpZGV0aWxkZXtcXHNpZ21hfV9QIiwyLHsiY3VydmUiOi00LCJzdHlsZSI6eyJ0YWlsIjp7Im5hbWUiOiJtYXBzIHRvIn19fV0sWzMsNCwiXFxwaGlfUCIsMix7InN0eWxlIjp7InRhaWwiOnsibmFtZSI6Im1hcHMgdG8ifX19XSxbNCwwLCJcXHNpZ21hX1AiLDIseyJzdHlsZSI6eyJ0YWlsIjp7Im5hbWUiOiJtYXBzIHRvIn19fV1d
\[\begin{tikzcd}
	&&& {z\otimes1+\lambda l+\lambda^pl'} \\
	\\
	{z\otimes1} & {z\otimes1} & {z\otimes1} & {z\otimes1+\varepsilon \overline{l}}
	\arrow["{\operatorname{id}\otimes \pi}"', maps to, from=1-4, to=3-4]
	\arrow["{\widetilde{\sigma}_P}"', curve={height=-24pt}, maps to, from=3-1, to=1-4]
	\arrow["{\operatorname{id}\otimes \pi}"', maps to, from=3-1, to=3-2]
	\arrow["{\phi_P}"', maps to, from=3-2, to=3-3]
	\arrow["{\sigma_P}"', maps to, from=3-3, to=3-4]
\end{tikzcd}\]

where $$\overline{m}\in R_P[\varepsilon],\ m\in R_P[\lambda],\ \overline{n_1}\in (y,z)R_P[\varepsilon],\ \overline{n_2}\in R_P[\varepsilon],$$ $$n_1\in (y,z)R_P[\lambda],\  n_2,n'\in R_P[\lambda],\  \overline{l}\in zR_P[\varepsilon], \ l\in zR_P[\lambda],\ l'\in R_P[\lambda]$$ such that $m, n_1, n_2, l$ are preimages of $\overline{m}, \overline{n_1},\overline{n_2},\overline{l}$, respectively.

In $R_P\otimes _k k[\lambda]$, we have
$$ \begin{aligned}0&=\widetilde{\sigma}_P(xz-y^p )
 \\
 &=(x\otimes 1 +\lambda m)(z\otimes 1 +\lambda l+\lambda^p l')-(y\otimes 1 + 1\otimes \lambda +\lambda n_1+\lambda^2n_2+ \lambda ^p n')^p.
 \end{aligned}$$

 Let $\pi': R_P\to k$ be the canonical projection. The latter element goes under $\pi'\otimes\operatorname{id} :R_P\otimes_kk[\lambda]\to k\otimes_kk[\lambda]$ to
 $$\begin{aligned} 
 &(0\otimes 1 +\lambda m')(0\otimes 1 +\lambda\cdot 0+\lambda^p l'')-(0\otimes 1 + 1\otimes \lambda +\lambda \cdot 0+ \lambda^2n_2' +\lambda ^p n'')^p\\
 =&-\lambda^p(1+\lambda n_2'+\lambda^{p-1}n'')^p \\
 =&-\lambda^p\neq 0 ,\\
\end{aligned}$$
hence contradiction. Here, we used the relation $\lambda^{p+1}=0$. $m',l'',n_2',n''$ are the images of $m,l',n_2,n'$ in $k\otimes_kk[\lambda]$. 
\end{proof}

\textbf{Acknowledgements.} The author would like to thank Professors Keiji Oguiso, Stefan Schröer, Xun Yu, Yujiro Kawamata and Gebhard Martin and the members of Oguiso's laboratory for their interest in this work and comments for preliminary versions.


\begin{thebibliography}{9}
\bibitem{BGL}
Bakker, B., Guenancia, H., Lehn, C.:
Algebraic approximation and the decomposition theorem for Kähler Calabi–Yau varieties. 
Invent. math. \textbf{228}, 1255–1308 (2022)
\bibitem{Fle}Flenner, H.; Kosarew, S.:  
On locally trivial deformations.
Publ. Res. Inst. Math. Sci. \textbf{23}, No. 4, 627–665  (1987)
\bibitem{Sch}
Schlessinger, M.: 
Functors of Artin rings.
Trans. Amer. Math. Soc. \textbf{130}, No. 2, 208–222 (Feb., 1968)
\bibitem{Ser}
Sernesi, E.: 
Deformations of Algebraic Schemes.
Grundlehren math. Wissenshaften \textbf{334}, Springer  (2006)




\end{thebibliography}
\end{document}